\documentclass{article}

\usepackage[utf8]{inputenc}
\usepackage[dvipsnames]{xcolor}
\usepackage[margin=1.25in]{geometry}
\usepackage[english]{babel}
\usepackage{amsfonts, amsmath, amscd, amsmath, amsthm, amssymb}
\usepackage[mathscr]{eucal}
\usepackage{indentfirst} 
\usepackage{graphics, graphicx}
\usepackage{fancyhdr}  
\usepackage{mathrsfs}
\usepackage{braket}  
\usepackage{array}
\usepackage{hyperref}
\usepackage{cleveref}
\usepackage{thmtools}
\usepackage{mathtools}
\usepackage{enumitem}
\usepackage[framemethod=TikZ]{mdframed}

\DeclareMathSymbol{\shortminus}{\mathbin}{AMSa}{"39}

\definecolor{gentlered}{HTML}{E76F64}

	{\end{mdframed}
}

\definecolor{creme}{HTML}{F2D6B4}
\mdfdefinestyle{cremetab}{%
	skipabove=4pt,
	linewidth=2pt,
	rightline=false,
	leftline=true,
	topline=false,
	bottomline=false,
	linecolor=creme!60!Black,
	backgroundcolor=creme!20,
}
\declaretheoremstyle[
	headfont=\bfseries\sffamily\color{creme!60!Black},
	bodyfont=\normalfont,
	spaceabove=0.5pt,
	spacebelow=0.5pt,
	mdframed={style=cremetab},
	headpunct={ --- },
]{definition}

\definecolor{softgreen}{HTML}{A9CBB8}
\mdfdefinestyle{greenrbox}{%
	roundcorner =10pt,
	linewidth=1pt,
	skipabove=12pt,
	innerbottommargin=9pt,
	skipbelow=2pt,
	linecolor=softgreen!80!black,
	nobreak=true,
	backgroundcolor=softgreen!10,
}
\declaretheoremstyle[
	headfont=\sffamily\bfseries\color{softgreen!80!black},
	mdframed={style=greenrbox},
	headpunct={\\[3pt]},
	postheadspace={0pt}
]{result}

\definecolor{fgblue}{HTML}{7696BA}
\definecolor{bgblue}{HTML}{9DC6DD}
\mdfdefinestyle{bluetab}{%
	skipabove=4pt,
	linewidth=2pt,
	rightline=false,
	leftline=true,
	topline=false,
	bottomline=false,
	linecolor=fgblue,
	backgroundcolor=bgblue!12,
}
\declaretheoremstyle[
	headfont=\bfseries\sffamily\color{fgblue},
	bodyfont=\normalfont,
	spaceabove=0.5pt,
	spacebelow=0.5pt,
	mdframed={style=bluetab},
	headpunct={ --- },
]{theorem}

\definecolor{nicegray}{HTML}{B5B7AE}
\mdfdefinestyle{graytab}{%
	skipabove=4pt,
	linewidth=2pt,
	rightline=false,
	leftline=true,
	topline=false,
	bottomline=false,
	linecolor=nicegray!90!black,
	backgroundcolor=nicegray!10,
}
\declaretheoremstyle[
	headfont=\bfseries\sffamily\color{nicegray!90!black},
	bodyfont=\normalfont,
	spaceabove=0.5pt,
	spacebelow=0.5pt,
	mdframed={style=graytab},
	headpunct={ --- },
]{corollary}

\definecolor{niceorange}{HTML}{EA9977}
\mdfdefinestyle{orangebox}{%
	linewidth=0.5pt,
	skipabove=12pt,
	frametitleaboveskip=5pt,
	frametitlebelowskip=0pt,
	skipbelow=2pt,
	frametitlefont=\bfseries,
	innertopmargin=4pt,
	innerbottommargin=8pt,
	nobreak=true,
	backgroundcolor=niceorange!5,
	linecolor=niceorange!85!black,
}
\declaretheoremstyle[
	headfont=\bfseries\color{niceorange!85!black},
	mdframed={style=orangebox},
	headpunct={\\[3pt]},
	postheadspace={0pt},
]{example}

\definecolor{nicepink}{HTML}{FADADD}
\mdfdefinestyle{pinktab}{%
	skipabove=4pt,
	skipbelow=4pt,
	linewidth=2pt,
	rightline=false,
	leftline=true,
	topline=false,
	bottomline=false,
	linecolor=nicepink!80!black,
	backgroundcolor=nicepink!25,
}
\declaretheoremstyle[
	headfont=\bfseries\sffamily\color{nicepink!80!black},
	bodyfont=\normalfont,
	spaceabove=0.5pt,
	spacebelow=0.5pt,
	mdframed={style=pinktab},
	headpunct={ --- },
]{lemma}

\theoremstyle{definition}

\Crefname{define}{Def.}{Defs.}
\Crefname{res}{Res.}{Res.}
\Crefname{theo}{Thm.}{Thms.}
\Crefname{cor}{Cor.}{Cors.}
\Crefname{ex}{Ex.}{Exs.}
\usepackage{subfiles}
\usepackage{enumitem}
\usepackage[utf8]{inputenc}
\usepackage[dvipsnames]{xcolor}
\usepackage[margin=1.25in]{geometry}
\usepackage[english]{babel}
\usepackage{amsfonts, amsmath, amscd, amsmath, amssymb}
\usepackage[mathscr]{eucal}
\usepackage{indentfirst} 
\usepackage{graphics, graphicx}
\usepackage{fancyhdr}  
\usepackage{mathrsfs}
\usepackage{braket}  
\usepackage{array}
\usepackage{hyperref}
\usepackage{cleveref,enumitem}
\usepackage{thmtools}
\usepackage{mathtools, tikz}
\usepackage{graphicx}
\usepackage{float}
\graphicspath{ {./images/} }
\usepackage{listings}
\usepackage[framemethod=TikZ]{mdframed}

\usepackage{amsthm}
\theoremstyle{plain}
\newtheorem{thm}{Theorem}[section] 
\newtheorem{defn}[thm]{Definition} 
\newtheorem{lmm}[thm]{Lemma} 
\newtheorem*{con*}{Conjecture}
\theoremstyle{remark}
\newtheorem{rem}{Remark} 

\usepackage[citestyle=numeric,bibstyle=numeric,backend=bibtex,giveninits=true,doi=false,eprint=false,isbn=false,maxnames=6]{biblatex}
\begin{document}
\begin{center}
    \Huge \textbf{Graphical Distances \& Inertia} 
    
    \vspace{0.75cm}
    
    \Large \textbf{Jeffrey Cheng\footnote{Department of Mathematics, The University of Texas at Austin, Austin, TX. email:  jeffrey.cheng@utexas.edu}, Charles Johnson\footnote{Department of Mathematics, The College of William \& Mary, Williamsburg, VA. email: crjohn@wm.edu}, Ian M.J. McInnis\footnote{Department of Mathematics, Princeton University, Princeton, NJ. email: imj@princeton.edu}, \& Matthew Yee\footnote{Department of Mathematics, Brown University, Providence, RI. email: matthew\_w\_yee@brown.edu}}
    
    \vspace{0.5cm}
    
    \normalsize \textbf{Abstract}
    
    \vspace{0.1cm}
    
    \normalsize \par We study the inertia of distance matrices of weighted graphs. Our novel congruence-based proof of the inertia of weighted trees extends to a proof for the inertia of weighted unicyclic graphs whose cycle is a triangle. Partial results are given on the inertia of other rationally weighted unicylic graphs. 
\end{center}

    \vspace{0.5cm}
    
    \normalsize \noindent \textbf{Keywords/phrases:} distance matrices of graphs, eigenvalues of graphs, euclidean distance matrices, inertia, matrix completion problems, unipositive matrices, weighted graphs 
    
    \vspace{0.1cm}
    
    
    

\section{Introduction}
Let $G$ be a simple, undirected, connected, edge-weighted graph on $n$ vertices, with vertices labeled $v_1, ..., v_n$. The \textit{distance matrix} $D=D(G)$ of $G$ is defined to be $D=[d_{ij}]$, where $d_{ij}$ is the minimum sum of the weights of any path from $v_i$ to $v_j$. The \textit{inertia} of a symmetric matrix $D$ is the ordered triple of integers: $n(D)=(n_+(D), n_0(D), n_-(D))$, where $n_+(D)$, $n_0(D)$, and $n_-(D)$ denote the number of positive, 0, and negative eigenvalues of $D$ respectively. 

\begin{rem} Note that the labeling of the vertices does not change the eigenvalues of $D$. Any two labelings differ by a permutation, and the eigenvalues of a square matrix are invariant under permutation. Thus, we may label vertices however we want. \end{rem}

This ``distance inertia'' $n(D(G))$ has been studied extensively. 
Early results concerned only the inertia of unweighted graphs. Graham and Pollack \cite{MR289210} showed that the inertia of an unweighted tree on $n$ vertices is $(1,0,n-1)$. Bapat, Kirkland, and Neumann \cite{MR2133282} extended this result to weighted trees and also characterized the inertia of unweighted unicyclic graphs. Zhang and Godsil \cite{MR3056981} characterized the inertia of unweighted cactus graphs. Further survey, along with some results on weighted graphs, appears in \cite{ChengMcInnisYee}.

\begin{rem} A \textit{cactus graph} is a connected graph of which all cycles are pairwise edge-disjoint.  Since trees and unicyclic graphs are special cases of cactus graphs, the results of \cite{MR3056981} subsume the results on unweighted graphs from \cite{MR289210} and \cite{MR2133282}. \end{rem}

\begin{defn} A matrix $D$ is \textbf{weakly unipositive} if $n_+(D)=1$. It is \textbf{strongly unipositive} if furthermore $n_0(D)=0$. Call a weighted graph $G$ strongly [weakly] unipositive when $D(G)$ is strongly [weakly] unipositive. A graph is \textbf{universally strongly [weakly] unipositive} if it is strongly [weakly] unipositive for any weighting. \end{defn}

\section{Inertia using Congruences}
We present a novel proof using congruences of the fact that trees are universally strongly unipositive (this fact was originally presented in \cite{MR2133282}). Furthermore, we extend this strategy to show that unicyclic graphs such that the graph is a 3-cycle are universally stongly unipositive as well. We begin by mentioning a couple of crucial theorems:
\begin{thm}
\textbf{(Sylvester's Law of Inertia)} Let $A \in M_n(\mathbb{R})$ be symmetric. Then, the given symmetric $B \in M_n(\mathbb{R})$, $A$ and $B$ have the same inertia $\iff$ $A=SBS^{T}$ for some non-singular $S$.
\end{thm}

\begin{thm}\label{Cauchy}
\textbf{(Cauchy Interlacing Theorem)}
Suppose $A \in M_n(\mathbb{R})$ is symmetric. Let $B \in M_m(\mathbb{R})$ with $m < n$ be a principal submatrix. Suppose $A$ has eigenvalues $\alpha_1, ..., \alpha_n$ and $B$ has eigenvalues $\beta_1, ..., \beta_m$. Then:
\begin{flalign*}
\alpha_k \leq \beta_k \leq \alpha_{k+n-m} \text{    for k=1, ..., m.}
\end{flalign*} 
\end{thm}

Using congruences, we are able to show that the distance matrix of a weighted tree has the same inertia as a "normal form"; namely, the distance matrix of any tree is congruent to a matrix of the form:
\[
D=\begin{bmatrix} 
    0 & d_1 & d_2 & \dots & d_{n-1}\\
    d_1 & -2d_1 & 0 & \dots & 0\\
    d_2 & 0 & -2d_2 & \dots & 0\\
    \vdots & \vdots & \vdots & \ddots & \vdots\\
    d_{n-1} & 0 & 0 & \dots & -2d_{n-1} 
    \end{bmatrix}
\]
The inertia of this normal form can be easily calculated. Using Sylvester's Law of Inertia, we obtain the desired results for the original graph.

We start by showing the congruence strategy in the case when the graph is a weighted star:
\begin{lmm}
The distance matrix of a weighted simple star on $n$ vertices has 1 positive eigenvalue and $(n-1)$ negative eigenvalues.
\end{lmm}

\begin{proof}
Let the center vertex of the star be $v_1$, the ordering of the other vertices does not matter. Now construct the matrix $D = [d_{ij}]$ such that $d_{ij}$ is the weighted path distance between vertices $i$ and $j$. Note that in a simple star, the weighted path distance between two non-central vertices $v_i$ and $v_j$ is $d_{1i}+d_{1j}$. Thus we have:
\[
D =\begin{bmatrix} 
    0 & d_{12} & d_{13} & \dots & d_{1n}\\
    d_{12} & 0 & d_{13}+d_{12} & \dots & d_{12}+d_{1n}\\
    d_{13} & d_{12}+d_{13} & 0 & \dots & d_{13}+d_{1n}\\
    \vdots & \vdots & \vdots & \ddots & \vdots\\
    d_{1n} & d_{1n}+d_{12} & d_{1n}+d_{13} & \dots & 0 
    \end{bmatrix}
\]
Now, consider the matrix:
\begin{flalign*}
A_{ij}=I-E(i,j)
\end{flalign*}
i.e. $A_{ij}$ is the matrix with 1's on the diagonal, -1 in the ij'th position, and 0's everywhere else. Clearly, each $A_{ij}$ is non-singular, so any product of $A_{ij}$ is also non-singular. Note that $A_{ij}B$ subtracts the $j$th row from the $i$th row of $B$ and $BA_{ij}^T$ subtracts the $j$th column from the $i$th column of $B$.
Let $A=\prod_{k=1}^nA_{k1}$. Then, we have the following matrix:
\[
AD =\begin{bmatrix} 
    0 & d_{12} & d_{13} & \dots & d_{1n}\\
    d_{12} & -d_{12} & d_{12} & \dots & d_{12}\\
    d_{13} & d_{13} & -d_{13} & \dots & d_{13}\\
    \vdots & \vdots & \vdots & \ddots & \vdots\\
    d_{1n} & d_{1n} & d_{1n} & \dots & -d_{1n} 
    \end{bmatrix}
\]
Note that $A^T=\prod_{k=1}^nA_{1k}$, so where $AB$ subtracts the first row from every other row in $B$, $BA^T$ subtracts the first column from every other column in $B$. Applying this to $AD$, we end up with:
\[
ADA^T =\begin{bmatrix} 
    0 & d_{12} & d_{13} & \dots & d_{1n}\\
    d_{12} & -2d_{12} & 0 & \dots & 0\\
    d_{13} & 0 & -2d_{13} & \dots & 0\\
    \vdots & \vdots & \vdots & \ddots & \vdots\\
    d_{1n} & 0 & 0 & \dots & -2d_{1n} 
    \end{bmatrix}
\]
Clearly, $D$ is congruent to $ADA^T$.

Thus, by Sylvester's Law of Inertia, it suffices to show that $ADA^T$ has 1 positive eigenvalue and (n-1) negative eigenvalues. This is easily done via an interlacing argument: firstly, notice that the $2\times 2$ leading principal submatrix of $ADA^T$ has eigenvalues $(2\sqrt{2}-1)d_{12}, -(2\sqrt{2}+1)d_{12}$. Thus, by the Cauchy Interlacing Theorem, it follows that $ADA^T$ has at least 1 positive eigenvalue. Now, consider the lower right $(n-1)\times (n-1)$ submatrix of $ADA^T$. It is a diagonal matrix, so it is easily seen to have eigenvalues $-2d_{12}, ..., -2d_{1n}$. Thus, by the Cauchy Interlacing Theorem, it follows that $ADA^T$ has n-1 negative eigenvalues. Thus, $ADA^T$ has 1 positive eigenvalue and $(n-1)$ negative eigenvalues and it follows that $ADA^T$ has 1 positive eigenvalue and $(n-1)$ negative eigenvalues.
\end{proof}

Notice that the desired congruences were obtained by "eliminating in from the pendents": a base vertex (in the case of the star, the central one) was chosen to be $v_1$, and then all other vertices were numbered arbitrarily. Then, the congruences were obtained by pulling back the pendents; the first row and column were subtracted from the row and column corresponding to the pendent vertex. 

It turns out that this idea still works when the tree has longer paths. We formalize this in the following theorem:

\begin{thm}\label{weightedtree}
The distance matrix of a weighted tree on $n$ vertices is strongly unipositive, that is to say, has 1 positive eigenvalue and $n-1$ negative eigenvalues.
\end{thm}

\begin{proof}
Let $T$ be a weighted tree on $n$ vertices. First, we define an ordering on the vertices. Consider a maximum unweighted path of T and let one of its end points be $v_1$. Label the rest of the vertices in the graph in an non-decreasing order of unweighted path length to $v_1$.
We can now construct the distance matrix $D=[d_{ij}]$, where $d_{ij}$ is the weighted path distance between vertices $i$ and $j$. Now, consider the matrix:
\begin{flalign*}
A_{ij}=I-E(i,j)
\end{flalign*}
i.e. $A_{ij}$ is the matrix with 1's on the diagonal, -1 in the ij'th position, and 0's everywhere else. Clearly, each $A_{ij}$ is non-singular. Note that $A_{ij}DA_{ij}^T$ subtracts the $j$th row from the $i$th row and subtracts the $j$th column from the $i$th column. Now, we have a chain of congruences on $D$ as follows:

Suppose that $T$ has $m$ leaves excluding $v_1$. For each leaf $v_i$, $n-m < i \leq n$, define $a(i)$ as the function which maps the index of a leaf to the index of vertex the leaf is adjacent to. Consider
\begin{flalign*}
A=\prod_{i~\mid~ v_i \text{ leaf}}A_{i,a(i)}
\end{flalign*}
where $A_{ij}$ is defined as in Lemma 1. Let $D_1=ADA^T$. Let $T_1$ be the weighted tree formed from deleting those $m$ leaves from $T$. Now, $T_1$ has $m_1$ leaves and this process may be applied to form $D_2 = A_1D_1A_1^T$, where $A_1$ is the matrix analogous to $A$ for $T_1$. Repeat this until $T_k = v_1$. Then, we claim that $D_k$ is of the form:
\[
\begin{bmatrix} 
    0 & d_1 & d_2 & \dots & d_{n-1}\\
    d_1 & -2d_1 & 0 & \dots & 0\\
    d_2 & 0 & -2d_2 & \dots & 0\\
    \vdots & \vdots & \vdots & \ddots & \vdots\\
    d_{n-1} & 0 & 0 & \dots & -2d_{n-1} 
    \end{bmatrix}
\]
where $d_i$ is the weight on the edge between $v_{i+1}$ and the adjacent vertex which has a smaller path distance to $v_1$. The same argument applied in Lemma 1 shows that this matrix, which is congruent to $D$, has 1 positive eigenvalue and (n-1) negative eigenvalues. Thus, so does $D$.

We will prove that $D_k$ is indeed of this given form through induction on the maximum unweighted path length of a weighted tree. Assume that this algorithm gives the desired form for any weighted tree such that every vertex has a path distance of $b$ or less from $v_1$. (Note that our base case, when $b=1$, means that the tree is a star, the case proven in Lemma 1.) Now, let us prove that this implies that our algorithm works for any weighted tree such that every vertex has a path distance of $b+1$ or less from $v_1$.
Let $T$ be such a weighted tree on n vertices with a distance matrix $D$. Moreover, let $L$ be the set of leaves of path distance $b+1$ from $v_1$. Given the ordering used in our algorithm, if there are $l$ such leaves, these are the points $v_{n+1-l}, v_{n+2-l},...,v_n$. Note that for a leaf $v_i$, its distance from any other point $v_j$ (except itself), is $d_{i,a(i)}+d_{a(i),j}$. Let us now define
\begin{flalign*}
A'=\prod_{i~\mid~ v_i \in L}A_{i,a(i)}~.
\end{flalign*}
Of course, $A'$ is a factor of the matrix $A$ which is used in the first step of the algorithm and it will subtract the $a(i)$th row from the $i$th row for every $v_i\in L$. If we apply $A'$ to $D$, then the $i$th row of $D$ will be filled with $d_{i,a(i)}$ except the $ii$th entry which will be $-d_{i,a(i)}$, for $n-l<i\leq n$ (the notation $d_{i,a(i)}=d_{i-1}$ will be used for legibility):
\[
A'D=\begin{bmatrix} 
    0 & d_{12} & \dots & d_{1,n+1-l} & d_{1,n+2-l} & \dots  & d_{1n}\\
    d_{12} & 0 & \dots & d_{2,n+1-l} & d_{2,n+2-l} & \dots & d_{2n} \\
    \vdots & \vdots & \ddots & \vdots & \vdots & \ddots & \vdots \\
    d_{n-l} & d_{n-l} & \dots & -d_{n-l} & d_{n-l} & \dots & d_{n-l}\\
    d_{n+1-l} &  d_{n+1-l} & \dots &  d_{n+1-l} &  -d_{n+1-l} & \dots &  d_{n+1-l} \\
    \vdots & \vdots & \ddots & \vdots & \vdots & \ddots & \vdots \\
    d_{n-1} & d_{n-1} & \dots & d_{n-1} & d_{n-1} & \dots & -d_{n-1}
    \end{bmatrix}
\]
Now, it is time to apply $A'^T$, which subtracts the $a(i)$th column from the $i$th column. Applying this to $A'D$, and using our familiar notation, we end up with:
\[
A'DA'^T=\begin{bmatrix} 
    0 & d_{12} & \dots & d_{1,n+1-l} & d_{1,n+2-l} & \dots  & d_{1n}\\
    d_{12} & 0 & \dots & d_{n-l} & d_{n+1-l} & \dots & d_{n-1} \\
    \vdots & \vdots & \ddots & \vdots & \vdots & \ddots & \vdots \\
    d_{n-l} & d_{n-l} & \dots & -2d_{n-l} & 0 & \dots & 0\\
    d_{n+1-l} &  d_{n+1-l} & \dots &  0 &  -2d_{n+1-l} & \dots &  0 \\
    \vdots & \vdots & \ddots & \vdots & \vdots & \ddots & \vdots \\
    d_{n-1} & d_{n-1} & \dots & 0 & 0 & \dots & -2d_{n-1}
    \end{bmatrix}
\]
Next, it is time to consider the upper left $(n-l)$x$(n-l)$ principal submatrix which has so far been left untouched. Notice that this submatrix is the distance matrix associated with the subtree of $T$ composed of the vertices with path distance $b$ or less from $v_1$. However, it is our inductive hypothesis that our algorithm returns the desired form for any weighted tree of vertices with path distance $b$ or less from a given $v_1$. So we may use our algorithm to ensure that the upper left block is in the correct form. Now, let us consider the impact of doing this on the matrix $D$ as a whole. Since the process will only subtract from rows and columns with indices less than or equal to $n-l$, the lower right $l$x$l$ will not change. However, note that every column in the lower left $l$x$(n-l)$ is the same. Thus, whenever we subtract the $a(i)$th column from the $i$th column for $1<i\leq n-l$, the $i$th column in this lower left block is zeroed. Similarly, every row in the upper right $(n-l)$x$l$ is the same. Thus, whenever we subtract the $a(i)$th row from the $i$th row for $1<i\leq n-l$, the $i$th row in this upper right block is zeroed. Our algorithm subtracts exactly one row from each row except for the first and subtracts exactly one column from each column except the first. Thus, in the upper right and lower left blocks, every entry is 0 except for those in the first row or column of $D$ once our algorithm on the upper left $(n-l)$x$(n-l)$ principal submatrix finishes: 
\[
D_k=\begin{bmatrix} 
    0 & d_{1} & \dots & d_{n-l} & d_{n+1-l} & \dots  & d_{1n}\\
    d_{1} & -2d_1 & \dots & 0 & 0 & \dots & 0 \\
    \vdots & \vdots & \ddots & \vdots & \vdots & \ddots & \vdots \\
    d_{n-l} & 0 & \dots & -2d_{n-l} & 0 & \dots & 0\\
    d_{n+1-l} & 0 & \dots &  0 &  -2d_{n+1-l} & \dots &  0 \\
    \vdots & \vdots & \ddots & \vdots & \vdots & \ddots & \vdots \\
    d_{n-1} & 0 & \dots & 0 & 0 & \dots & -2d_{n-1}
    \end{bmatrix}
\]
Since this is exactly the form we claimed and the case $b=1$ works, we have proven that the distance matrix of a weighted tree of arbitrary size has 1 positive eigenvalue and (n-1) negative eigenvalues.
\end{proof}

As stated above, this method of congruences can be extended to show that more complicated graphs are also strongly unipositive. In these cases, the "normal form" that we receive after congruence is slightly more complicated, so more sophisticated machinery is needed to calculate the determinant. We showcase the required theorem below:

\begin{thm}
\textbf{(Schur Determinant Formula)} Let $M$ be a square matrix that can be partitioned in the form:
\[
M=\begin{bmatrix}
A & B \\
C & D \\
\end{bmatrix}
\]
Then, if $A$ is non-singular:
\[
det(M)=det(A)det(D-CA^{-1}B)
\]
\end{thm}

We are now ready to state the main results, in the same manner as before (starting with the star, then moving to the general case). Once again, the proof boils down to calculating the determinant of the "normal form". For the star case, we induct on the number of pendent vertices of the star (with the base case being 3 pendent vertices, as any less is trivial).

\begin{lmm}\label{star4} The distance matrix for the star on 4 vertices + an edge has one positive eigenvalue and 3 negative eigenvalues.
\end{lmm}
\begin{proof}
As above, let the center vertex of the star be $v_1$, However,in this case, let $v_2$, $v_3$ be the two vertices of the original star connected by the extra edge. Now construct the matrix $D = [d_{ij}]$ such that $d_{ij}$ is the weighted path distance between vertices $i$ and $j$. Note that in a simple star, the weighted path distance between two non-central vertices $v_i$ and $v_j$ is $d_{1i}+d_{1j}$. If $d_{23}$ is redundant, then the distance matrix is the same as for a simple star and the result follows by the above. Thus, we assume that $d_{23}$ is not redundant and we have:
\[
D =\begin{bmatrix} 
    0 & d_{12} & d_{13} & d_{14}\\
    d_{12} & 0 & d_{23} & d_{12}+d_{14}\\
    d_{13} & d_{23} & 0 & d_{13}+d_{14}\\
    d_{14} & d_{14}+d_{12} & d_{14}+d_{13} & 0 
    \end{bmatrix}
\]
Now, consider the matrix:
\begin{flalign*}
A_{ij}=I-E(i,j)
\end{flalign*}
i.e. $A_{ij}$ is the matrix with 1's on the diagonal, -1 in the ij'th position, and 0's everywhere else. Clearly, each $A_{ij}$ is non-singular, so any product of $A_{ij}$ is also non-singular Note that $A_{ij}B$ subtracts the $j$th row from the $i$th row of $B$ and $BA_{ij}^T$ subtracts the $j$th column from the $i$th column of $B$.
Let $A=\prod_{k=1}^3A_{k1}$. By applying the same congruence as above, we obtain:
\[
ADA^T =\begin{bmatrix} 
    0 & d_{12} & d_{13} & d_{14}\\
    d_{12} & -2d_{12} & d_{23}-d_{12}-d_{13} & 0\\
    d_{13} & d_{23}-d_{12}-d_{13} & -2d_{13} & 0\\
    d_{14} & 0 & 0 & -2d_{14}
    \end{bmatrix}
\]
Clearly, $D$ is congruent to $ADA^T$.

Thus, by Sylvester's Law of Inertia, it suffices to show that $ADA^T$ has 1 positive eigenvalue and 3 negative eigenvalues. As above, this is done via an interlacing argument: firstly, notice that the 2 x 2 leading principal submatrix of $ADA^T$ has eigenvalues $(2\sqrt{2}-1)d_{12}, -(2\sqrt{2}+1)d_{12}$. Thus, by the Cauchy Interlacing Theorem, it follows that $ADA^T$ has at least 1 positive eigenvalue. Now, consider the lower right 3 x 3 submatrix of $ADA^T$ (call this $B$). It suffices to show that all eigenvalues of this matrix are negative. To show this, it suffices to show that the determinant of the entire matrix is negative. A straight determinant calculation shows that this is the case: let $z=d_{23}-d_{12}-d_{13}$. Notice that as $d_{23} \in (|d_{12}-d_{13}|,d_{12}+d_{13})$, it easily follows that $|z| \in (0, 2min(d_{12},d_{13}))$. Then:
\[
det(B)=-2d_{13}(4d_{12}d_{13}-z^2) < 0 \iff (4d_{12}d_{13}-z^2) > 0
\].
However, this is easily seen via the Intermediate Value Theorem, as $4d_{12}d_{23} > 0$ and $4d_{12}d_{13}-(2min(d_{12},d_{13}))^2 > 0$, and $f(z)=4d_{12}d_{13}-z^2$ is a strictly decreasing continuous function on $(0, 2min(d_{12},d_{13}))$. Thus, the determinant is negative and $B$ has either 3 negative eigenvalues or 2 positive eigenvalues and 1 negative eigenvalue. However, if $B$ has 2 positives and 1 negative, this is a contradiction: as $B$ is congruent to a hollow symmetric nonnegative matrix (namely, the lower right 3 x 3 submatrix of D), this is impossible (if it were the case, we would have that the negative eigenvalue had the largest absolute value, a contradiction to Perron-Frobenius). 

Thus, by the Cauchy Interlacing Theorem, it follows that $ADA^T$ has 3 negative eigenvalues. Thus, $ADA^T$ has 1 positive eigenvalue and 3 negative eigenvalues and it follows that $ADA^T$ has 1 positive eigenvalue and 3 negative eigenvalues. 
\end{proof}

As the base case has been established, we can now use induction to show the result for a star with any number of vertices + an edge.

\begin{lmm}\label{weightedstaredge}
The distance matrix of a simple star with a single additional edge between two non-central vertices is strongly unipositive.
\end{lmm}

\begin{proof}
We proceed by induction on the number of vertices of the star.

Consider the base case: n=4. (The cases n=1,2,3 are not so interesting and can be easily shown in other elementary methods). Then, it follows exactly by Lemma \ref{star4} above. Now, assume that the statement is true for the star on $n$ vertices. We prove it for the star on $(n+1)$ vertices. Assign $v_1, ..., v_n$ as follows: let $v_1$ be the center of the star on $n$ vertices, $v_2, v_3$ be the vertices connected by the extra edge, all other vertices may be chosen as desired. Let $v_{n+1}$ be the new vertex added to the tree+edge on $n$ vertices. Let $D$ be the distance matrix of the graph on $n+1$ vertices, $D'$ be the distance matrix of the graph of the tree+edge on $n$ vertices. Then, after congruences (eliminating in from the pendents as above) we obtain:
\[
ADA^T=\begin{bmatrix} 
    0 & d_{12} & d_{13} & d_{14} & \dots  & d_{1n+1}\\
    d_{12} & -2d_{12} & d_{23}-d_{12}-d_{13} & 0 & \dots & 0 \\
    d_{13} & d_{23}-d_{12}-d_{13} & -2d_{13} & 0 & \dots & 0\\
    d_{14} & 0 &  0 &  -2d_{14} & \dots &  0 \\
    \vdots & \vdots & \vdots & \vdots & \ddots & \vdots \\
    d_{1n+1} & 0 & 0 & 0 & \dots & -2d_{1n+1}
    \end{bmatrix}
\]
Notice that the upper $n\times n$ principal submatrix of $ADA^T$-call this $X$-is exactly congruent to $D'$. Thus, it has one positive eigenvalue and (n-1) negative eigenvalues. We break following argument into two cases: when n is even, and when n is odd.

Firstly, assume that $n$ is even (so that the determinant of $D'$, which has the same sign as the determinant of $X$, is negative). Thus, by interlacing, $ADA^T$ has at least one positive eigenvalue and (n-1) negative eigenvalues and it suffices to show that $det(ADA^T)$ is positive. This is done via the Schur complement: notice that we can express $ADA^T$ as a block matrix:
\[
ADA^T=\begin{bmatrix}
    X & Y \\
    Z & L \\
\end{bmatrix}
\]
where $X$ is as above, $Z=\begin{bmatrix}d_{1n+1} & 0 & \dots & 0 \end{bmatrix}$, $Y=Z^T$, and $L=\begin{bmatrix} -2d_{1n+1} \end{bmatrix}$. Via the Schur complement formula, we have:
\[
det(ADA^T)=det(X)det(L-ZX^{-1}Y)
\]
(this holds as long as $X$ is nonsingular, but we assume that $det(X) \neq 0$ via the induction hypothesis). As $det(X) < 0$, $det(ADA^T) > 0$ $\iff$ $det(L-ZX^{-1}Y) < 0$. It suffices to show that $ZX^{-1}Y > 0 \iff d_{1n+1}^2X^{-1}_{11} > 0 \iff X^{-1}_{11} > 0$. Further, we have a formula for $X^{-1}$ given via the adjugate:
\[
X^{-1}=\frac{1}{det(X)}Adj(X)^T \implies X^{-1}_{11}=\frac{1}{det(X)}Adj(X)_{11}
\]
Thus, $X^{-1}_{11} > 0 \iff Adj(X)_{11} < 0$. However, this is easily seen to be true: as n is assumed to be even, n-1 is odd and it suffices to show that the determinant of $K$=the $(n-1)\times (n-1)$ lower right submatrix of $X$=the $(n-1)\times (n-1)$ central submatrix of $ADA^T$ is negative.
\[
K=\begin{bmatrix} 
    -2d_{12} & d_{23}-d_{12}-d_{13} & 0 & \dots & 0 \\
    d_{23}-d_{12}-d_{13} & -2d_{13} & 0 & \dots & 0\\
    0 &  0 &  -2d_{14} & \dots &  0 \\
    \vdots & \vdots & \vdots & \ddots & \vdots \\
    0 & 0 & 0 & \dots & -2d_{1n}
    \end{bmatrix}
\]
Notice that this is clearly negative: indeed, by taking the determinant by expanding along the last column, we obtain that $det(J)=-2d_{1n-1}...-2d_{14}(4d_{12}d_{13}-z^2)$, where $z$ is defined as in the lemma above. Note that since $n$ is even, $n-1-2=n-3$ is odd and the determinant is the product of an odd number of negative values and a positive value and is thus negative.

Now, assume $n$ is odd (so that the determinant of $X$ is positive). Then, it suffices to show that $det(ADA^T)$ is negative. As above, this is done via the Schur complement: notice that we can express $ADA^T$ as a block matrix:
\[
ADA^T=\begin{bmatrix}
    X & Y \\
    Z & L \\
\end{bmatrix}
\]
where $X$ is as above, $Z=\begin{bmatrix}d_{1n+1} & 0 & \dots & 0 \end{bmatrix}$, $Y=Z^T$, and $L=\begin{bmatrix} -2d_{1n+1} \end{bmatrix}$. Via the Schur complement formula, we have:
\[
det(ADA^T)=det(X)det(L-ZX^{-1}Y)
\]
As $det(X) > 0$, $det(ADA^T) < 0$ $\iff$ $det(L-ZX^{-1}Y) < 0$. It suffices to show that $ZX^{-1}Y > 0 \iff d_{1n+1}^2X^{-1}_{11} > 0 \iff X^{-1}_{11} > 0$. Further, we have a formula for $X^{-1}$ given via the adjugate:
\[
X^{-1}=\frac{1}{det(X)}Adj(X)^T \implies X^{-1}_{11}=\frac{1}{det(X)}Adj(X)_{11}
\]
Thus, $X^{-1}_{11} > 0 \iff Adj(X)_{11} > 0$. However, this is easily seen to be true: as n is assumed to be odd, n-1 is even and it suffices to show that the determinant of $K$=the $(n-1)\times (n-1)$ lower right submatrix of $X$=the $(n-1)\times (n-1)$ central submatrix of $ADA^T$ is positive.
\[
K=\begin{bmatrix} 
    -2d_{12} & d_{23}-d_{12}-d_{13} & 0 & \dots & 0 \\
    d_{23}-d_{12}-d_{13} & -2d_{13} & 0 & \dots & 0\\
    0 &  0 &  -2d_{14} & \dots &  0 \\
    \vdots & \vdots & \vdots & \ddots & \vdots \\
    0 & 0 & 0 & \dots & -2d_{1n}
    \end{bmatrix}
\]
Notice that this is clearly positive: indeed, by taking the determinant by expanding along the last column, we obtain that $det(J)=-2d_{1n-1}...-2d_{14}(4d_{12}d_{13}-z^2)$, where $z$ is defined as in the lemma above. Note that since $n$ is odd, $n-1-3=n-3$ is even and the determinant is the product of an even number of negative values and a positive value and is thus positive. The proof is complete.
\end{proof}

Similarly to the case for trees, it is easy to add pendent vertices to the star case once it has been established. Thus, we have the main result:
\begin{thm}
The distance matrix of any weighted unicyclic graph on n vertices such that the cycle is a 3-cycle is strongly unipositive.
\end{thm}
\begin{proof}
Let $T$ be a weighted unicyclic graph on $n$ vertices such that the cycle is a 3-cycle. First, we define an ordering on the vertices. Let the vertices of the cycle be $v_1$, $v_2$, and $v_3$. Label the rest of the vertices in the graph in an non-decreasing order of unweighted path length to $v_1$.
We can now construct the distance matrix $D=[d_{ij}]$, where $d_{ij}$ is the weighted path distance between vertices $i$ and $j$. Now, consider the matrix:
\begin{flalign*}
A_{ij}=I-E(i,j)
\end{flalign*}
i.e. $A_{ij}$ is the matrix with 1's on the diagonal, -1 in the ij'th position, and 0's everywhere else. Clearly, each $A_{ij}$ is non-singular. Note that $A_{ij}DA_{ij}^T$ subtracts the $j$th row from the $i$th row and subtracts the $j$th column from the $i$th column. Now, we have a chain of congruences on $D$ as follows:

Suppose that $T$ has $m$ leaves. For each leaf $v_i$, $n-m < i \leq n$, define $a(i)$ as the function which maps the index of a leaf to the index of vertex the leaf is adjacent to. Consider
\begin{flalign*}
A=\prod_{i~\mid~ v_i \text{ leaf}}A_{i,a(i)}
\end{flalign*}
where $A_{ij}$ is defined as in Lemma 1. Let $D_1=ADA^T$. Let $T_1$ be the weighted unicylclic graph formed from deleting those $m$ leaves from $T$. Now, $T_1$ has $m_1$ leaves and this process may be applied to form $D_2 = A_1D_1A_1^T$, where $A_1$ is the matrix analogous to $A$ for $T_1$. Repeat this until $T_{k-1}$ is the 3-cycle. Then, consider $v_2$ and $v_3$ to be pendent vertices as far as our algorithm is concerned and subtract the first row and first column from the second and third rows and second and third columns to get $D_k$. We claim that $D_k$ is of the form:
\[
\begin{bmatrix} 
    0 & d_{12} & d_{13} & d_{3} & \dots & d_{n-1}\\
    d_{12} & -2d_{12} & d_{23}-d_{12}-d_{13} & 0 & \dots & 0\\
    d_{13} & d_{23}-d_{12}-d_{13} & -2d_{13} & 0 & \dots & 0\\
    \\ d_3 & 0 & 0 & -2d_3 & \dots & 0 \\
    \vdots & \vdots & \vdots & \vdots & \ddots & \vdots\\
    d_{n-1} & 0 & 0 & 0 & \dots & -2d_{n-1} 
    \end{bmatrix}
\]
where $d_i$ is the weight on the edge between $v_{i+1}$ and the adjacent vertex which has a smaller path distance to $v_1$. The same argument applied in Lemma 3.7 shows that this matrix, which is congruent to $D$, has 1 positive eigenvalue and (n-1) negative eigenvalues. Thus, so does $D$.

We will prove that $D_k$ is indeed of this given form through induction on the maximum unweighted path length from $v_1$. Assume that this algorithm gives the desired form for any weighted unicyclic graph such that the cycle is a 3-cycle and such that every vertex has a path distance of $b$ or less from $v_1$. (Note that our base case, when $b=1$, means that the graph is a star with an extra edge, the case proven in Lemma 3.7.) Now, let us prove that this implies that our algorithm works for any weighted 3-unicylic graph such that every vertex has a path distance of $b+1$ or less from $v_1$.
Let $T$ be such a weighted 3-unicyclic graph on n vertices with a distance matrix $D$. Moreover, let $L$ be the set of leaves of path distance $b+1$ from $v_1$. Given the ordering used in our algorithm, if there are $l$ such leaves, these are the points $v_{n+1-l}, v_{n+2-l},...,v_n$. Note that for a leaf $v_i$, its distance from any other point $v_j$ (except itself), is $d_{i,a(i)}+d_{a(i),j}$. Let us now define
\begin{flalign*}
A'=\prod_{i~\mid~ v_i \in L}A_{i,a(i)}~.
\end{flalign*}
Of course, $A'$ is a factor of the matrix $A$ which is used in the first step of the algorithm and it will subtract the $a(i)$th row from the $i$th row for every $v_i\in L$. If we apply $A'$ to $D$, then the $i$th row of $D$ will be filled with $d_{i,a(i)}$ except the $ii$th entry which will be $-d_{i,a(i)}$, for $n-l<i\leq n$ (the notation $d_{i,a(i)}=d_{i-1}$ will be used for legibility):
\[
A'DA'^T=\begin{bmatrix} 
    0 & d_{12} & d_{13} & d{14} & \dots & d_{n-l} & d_{n-l+1} & \dots  & d_{n-1}\\
    d_{12} & 0 &  d_{23}& d_{24} & \dots & d_{n-l} & d_{n+1-l} & \dots & d_{n-1} \\
    d_{13} & d_{23} & 0 & d_{34} & \dots & d_{n-l} & d_{n+1-l} & \dots & d_{n-1} \\
    d_{14} & d_{24} &  d_{34}& 0 & \dots & d_{n-l} & d_{n+1-l} & \dots & d_{n-1} \\
    \vdots & \vdots & \vdots & \vdots & \ddots & \vdots & \vdots & \ddots & \vdots \\
    d_{n-l} & d_{n-l} & d_{n-l}& d_{n-l}& \dots & -2d_{n-l} & 0 & \dots & 0\\
    d_{n+1-l} &  d_{n+1-l} &d_{n+1-l} &  d_{n+1-l} & \dots &  0 &  -2d_{n+1-l} & \dots &  0 \\
    \vdots & \vdots &\vdots & \vdots & \ddots & \vdots & \vdots & \ddots & \vdots \\
    d_{n-1} & d_{n-1} & d_{n-1} & d_{n-1} & \dots & 0 & 0 & \dots & -2d_{n-1}
    \end{bmatrix}.
\]
Next, it is time to consider the upper left $(n-l)$x$(n-l)$ principal submatrix which has so far been left untouched. Notice that this submatrix is the distance matrix associated with the subgraph of $T$ composed of the vertices with path distance $b$ or less from $v_1$. However, it is our inductive hypothesis that our algorithm returns the desired form for any weighted 3-unicylic graph of vertices with path distance $b$ or less from a given $v_1$. So we may use our algorithm to ensure that the upper left block is in the correct form. Now, let us consider the impact of doing this on the matrix $D$ as a whole. Since the process will only subtract from rows and columns with indices less than or equal to $n-l$, the lower right $l$x$l$ will not change. However, note that every column in the lower left $l$x$(n-l)$ is the same. Thus, whenever we subtract the $a(i)$th column from the $i$th column for $1<i\leq n-l$, the $i$th column in this lower left block is zeroed. Similarly, every row in the upper right $(n-l)$x$l$ is the same. Thus, whenever we subtract the $a(i)$th row from the $i$th row for $1<i\leq n-l$, the $i$th row in this upper right block is zeroed. Our algorithm subtracts exactly one row from each row except for the first and subtracts exactly one column from each column except the first. Thus, in the upper right and lower left blocks, every entry is 0 except for those in the first row or column of $D$ once our algorithm on the upper left $(n-l)$x$(n-l)$ principal submatrix finishes: 
\[
D_k=\begin{bmatrix} 
    0 & d_{12} & d_{13} & d_{14} & \dots & d_{n-l} & d_{n-l+1} & \dots  & d_{n-1}\\
    d_{12} & -2d_{12} &  d_{23}-d_{12}-d_{13} & 0 & \dots & 0 & 0 & \dots & 0 \\
    d_{13} & d_{23}-d_{12}-d_{13} & -2d_{13} & 0 & \dots & 0 & 0 & \dots & 0 \\
    d_{14} & 0 &  0& -2d_{14} & \dots & 0 & 0 & \dots & 0 \\
    \vdots & \vdots & \vdots & \vdots & \ddots & \vdots & \vdots & \ddots & \vdots \\
    d_{n-l} & 0 & 0& 0& \dots & -2d_{n-l} & 0 & \dots & 0\\
    d_{n+1-l} &  0 &0& 0 & \dots &  0 &  -2d_{n+1-l} & \dots &  0 \\
    \vdots & \vdots &\vdots & \vdots & \ddots & \vdots & \vdots & \ddots & \vdots \\
    d_{n-1} & 0& 0 & 0 & \dots & 0 & 0 & \dots & -2d_{n-1}
    \end{bmatrix}.
\]
Since this is exactly the form we claimed and the case $b=1$ works, we have proven that the distance matrix of a weighted 3-unicylic graph of arbitrary size has 1 positive eigenvalue and (n-1) negative eigenvalues.
\end{proof}
\section{Inertia of Rationally Weighted Unicylic Graphs}
Next, we desired to characterize the inertia of weighted unicyclic graphs with larger cycles. Unfortunately, it becomes hard to extend the congruence strategy to these cases. However, we were able to obtain some partial results by extending the results of \cite{MR2133282} using an edge subdivision argument.
\begin{defn}
Given a graph $G=(V,E)$ with vertex set $V$ and edge set $E$, \textbf{edge subdvision} of an edge $\{u,v\} \ \in \ E$ is the operation of deleting $\{u,v\}$ from $E$, adding a new vertex $w$ to $V$, and adding the edges $\{u,w\}$ and $\{w,v\}$ to $E$. In the case of a weighted graph, the new edges will be given weights that sum to the weight of $\{u,v\}$.
\end{defn}
Roughly, the edge subdivision argument is as follows: given an appropriate weighted graph $G$, we can edge subdivide the graph until we receive an unweighted graph $G'$. Then, the distance matrix of the original graph will be a principal submatrix of the distance matrix of $G'$. Using results on the inertia of $G'$ and the Cauchy Interlacing Thereom, we are able to obtain results on the inertia of $G$. 

Recall the following theorems from \cite{MR2133282}:
\begin{thm}
Let $G$ be an unweighted unicyclic graph with $2k+1+m$ vertices and cycle length $2k+1$. Then the inertia of the shortest path distance matrix $D(G)$ is $(n_+(D), n_0(D), n_-(D))=(1,0,2k+m)$. That is to say, $D$ is strongly unipositive. 
\end{thm}
\begin{thm}
Let $G$ be an unweighted unicyclic graph with $2k+m$ vertices and cycle length $2k$. Then the inertia of the shortest path distance matrix $D(G)$ is $(n_+(D), n_0(D), n_-(D))=(1,k-1,k+m)$.
\end{thm}

We need to define a proper setting for our edge subdivision strategy to be applied. The following definitions characterize unicyclic graphs to which this can be applied:

\begin{defn}
We say that a cycle $C$ has \textbf{odd path-length sum} if the following conditions hold:
\begin{enumerate}
    \item The sum of the edges of the cycle is $\frac{a}{b}$, $a,b \in \mathbb{Z}$ in simplest form.
    \item If any of the weights are non-integers, then $a=2k+1$ for $k \in \mathbb{Z}$.
    \item If all the weights are integers, then $\frac{a}{c}=2k+1$ for $k \in \mathbb{Z}$ where $c$ is the greatest common factor among the edge weights.
\end{enumerate}
\end{defn}
\begin{rem}
"Even path-length sum" is defined analagously.
\end{rem}

Given a rationally weighted unicylic graph with odd/even-path length sum, we obtain the following results using the edge subdivision argument:

\begin{thm}
Let D be the distance matrix of a rationally weighted unicyclic graph $G$ on $n$ vertices such that the cycle has odd path-length sum. Then D has 1 positive eigenvalue and (n-1) negative eigenvalues, so D is strongly unipositive.
\end{thm}
\begin{proof}
Let $G$ be a rationally weighted unicyclic graph on $n$ vertices and let $C$ be its cycle with at least one edge whose weight is not an integer and let $D$ be the shortest path distance matrix of $C$. If the total path edge weight of $C$ is $\frac{a}{b}$, $a,b \in \mathbb{Z}$ and $a=2k+1$ for $k \in \mathbb{Z}$, then examine $bD$. $bD$ must have the same inertia as $D$ and it is the shortest path matrix of $bG$ where $bG$ is the graph formed by multiplying very edge weight in $G$ by $b$. However $bD$ is also a principal submatrix of the shortest path matrix of an unweighted unicyclic graph with a cycle of length $2k+1$., specifically the matrix of the graph formed from subdividing each edge of $bD$ into edges with weight $1$. Let us call this graph $(bG)'$ and its matrix $(bD)'$. Since $(bD)'$ has 1 positive eigenvalue and the rest negative eigenvalues, by Sylvester's law of inertia, $bD$ and thus $D$ may have one of three possible inertias: $(1,0,n-1)$, $(0,1,n-1)$, $(0,0,n)$. However, since $D$ has a trace of 0, the sum of its eigenvalues must be 0, so it is known that $D$ must have at least 1 positive eigenvalue. Thus, the inertia of $D$ must be $(n_+(D), n_0(D), n_-(D))=(1,0,n-1)$.

Instead iff all the weights are integers and $\frac{a}{c}=2k+1$ for $k \in \mathbb{Z}$ where $a$ is the total weighted path sum of $C$ and $c$ is the greatest common factor among the edge weights, then instead examine $\frac{1}{c}D$. This will be a principal submatrix of the shortest path matrix of an unweighted unicyclic graph with a cycle of length $2k+1$, specifically the matrix of the graph formed from subdividing each edge of $\frac{1}{c}D$ into edges with weight $1$. From here, the case is identical.
\end{proof}
\begin{thm}
Let D be the distance matrix of a rationally weighted unicyclic graph $G$ on $n$ vertices such that the cycle has even path-length sum. Then, D has 1 positive eigenvalue, $\leq k-1$ zero eigenvalues, and the rest negative eigenvalues.
\end{thm}
\begin{proof}
Let $G$ be a rationally weighted unicyclic graph on $n$ vertices and let $C$ be its cycle with at least one edge whose weight is not an integer and let $D$ be the shortest path distance matrix of $C$. If the total path edge weight of $C$ is $\frac{a}{b}$, $a,b \in \mathbb{Z}$ and $a=2k$ for $k \in \mathbb{Z}$, then examine $bD$. $bD$ must have the same inertia as $D$ and it is the shortest path matrix of $bG$ where $bG$ is the graph formed by multiplying very edge weight in $G$ by $b$. However $bD$ is also a principal submatrix of the shortest path matrix of an unweighted unicyclic graph with a cycle of length $2k$., specifically the matrix of the graph formed from subdividing each edge of $bD$ into edges with weight $1$. Let us call this graph $(bG)'$ and its matrix $(bD)'$. Since $(bD)'$ has 1 positive eigenvalue, the zero eigenvalue with multiplicity $k-1$, and the rest negative eigenvalues, by Sylvester's law of inertia, $bD$ and thus $D$ may not have greater than $k-1$ zero eigenvalues. However, since $D$ has a trace of 0, the sum of its eigenvalues must be 0, so it is known that $D$ must have at least 1 positive eigenvalue. Thus, $D$ has 1 positive eigenvalue, $\leq k-1$ zero eigenvalues, and the rest negative eigenvalues.

Instead iff all the weights are integers and $\frac{a}{c}=2k+1$ for $k \in \mathbb{Z}$ where $a$ is the total weighted path sum of $C$ and $c$ is the greatest common factor among the edge weights, then instead examine $\frac{1}{c}D$. This will be a principal submatrix of the shortest path matrix of an unweighted unicyclic graph with a cycle of length $2k+1$, specifically the matrix of the graph formed from subdividing each edge of $\frac{1}{c}D$ into edges with weight $1$. From here, the case is identical.
\end{proof}



\newpage
\printbibliography[
heading=bibintoc,
title={References}
]
\end{document}